\documentclass[10pt,reqno]{amsart}

\usepackage[margin=1in]{geometry}  
\usepackage{amsmath}
\usepackage{amssymb}
\usepackage[pdftex]{graphicx}
\usepackage{verbatim}
\usepackage{color}
\usepackage{mathrsfs}
\usepackage[colorlinks=true]{hyperref}
\hypersetup{urlcolor=blue, citecolor=blue, pdfstartview=FitH}
\hypersetup{pdfstartview={XYZ null null 1.25}}
\usepackage{fontenc}
\usepackage{textcomp}
\usepackage{fix-cm}
\usepackage{array}
\usepackage{enumerate}
\usepackage{latexsym, amsfonts, amssymb}
\usepackage{color}
\usepackage{txfonts}
\usepackage{comment}
\usepackage{subfigure}
\usepackage{accents}

\numberwithin{equation}{section}

\newtheorem{theorem}{Theorem}[section]

\newtheorem{lemma}[theorem]{Lemma}
\newtheorem{corollary}[theorem]{Corollary}

\newtheorem{remark}[theorem]{Remark}

\def\XXint#1#2#3{{\setbox0=\hbox{$#1{#2#3}{\int}$ }
\vcenter{\hbox{$#2#3$ }}\kern-.555\wd0}}

\begin{document}
\fontsize{11}{13}\selectfont

\title[Stagnation-point Form Solutions of 2D Boussinesq]
{Blowup in Stagnation-point Form Solutions of the Inviscid\\ 2d Boussinesq Equations}

\author{Alejandro Sarria and Jiahong Wu}

\address{Department of Mathematics \\
University of Colorado at Boulder \\
Boulder, CO 80309-0395 USA \\
}
\email{alejandro.sarria@colorado.edu}

\address{Department of Mathematics \\ Oklahoma State University \\ Stillwater, OK 74078 USA}
\email{jiahong.wu@okstate.edu}




\begin{abstract}
The 2d Boussinesq equations model large scale atmospheric and oceanic flows. Whether its solutions develop a singularity in finite-time remains a classical open problem in mathematical fluid dynamics. In this work, blowup from smooth nontrivial initial velocities in stagnation-point form solutions of this system is established. On an infinite strip $\Omega=\{(x,y)\in[0,1]\times\mathbb{R}^+\}$, we consider velocities of the form $u=(f(t,x),-yf_x(t,x))$, with scalar temperature\, $\theta=y\rho(t,x)$. Assuming $f_x(0,x)$ attains its global maximum only at points $x_i^*$ located on the boundary of $[0,1]$, general criteria for finite-time blowup of the vorticity $-yf_{xx}(t,x_i^*)$ and the time integral of $f_x(t,x_i^*)$ are presented. Briefly, for blowup to occur it is sufficient that $\rho(0,x)\geq0$ and $f(t,x_i^*)=\rho(0,x_i^*)=0$, while $-yf_{xx}(0,x_i^*)\neq0$. To illustrate how vorticity may suppress blowup, we also construct a family of global exact solutions. A local-existence result and additional regularity criteria in terms of the time integral of $\left\|f_x(t,\cdot)\right\|_{L^\infty([0,1])}$ are also provided.

\end{abstract}

\maketitle

\section{Introduction}
\label{sec:intro}

In this article we discuss regularity criteria for solutions of the initial value problem
\begin{equation}
\label{reduced}
\begin{cases}
f_{xt}+ff_{xx}-f_x^{2}+\rho=I(t),\qquad\quad\qquad\qquad &x\in[0,1],\,\, t>0,
\\
\rho_t+f\rho_x=\rho f_x\,, \qquad\quad &x\in[0,1],\,\, t>0,
\\
I(t)=\int_0^1{\rho\,dx}-2\int_0^1{f_x^2\,dx},\qquad &t>0,
\\
f(x,0)=f_0(x),\,\,\,\rho(x,0)=\rho_0(x),\qquad &x\in[0,1],
\end{cases}
\end{equation}
subject to either Dirichlet
\begin{equation}
\label{dbc}
\begin{split}
\qquad\qquad\qquad f(t,0)=f(t,1)=0,\qquad\quad\quad\quad \rho(t,0)=\rho(t,1)=0,
\end{split}
\end{equation}
or periodic boundary conditions
\begin{equation}
\label{pbc}
\begin{split}
f(t,0)=f(t,1),\quad\,\,\,f_x(t,0)=f_x(t,1),\quad\,\,\,\qquad\rho(t,0)=\rho(t,1).
\end{split}
\end{equation}

System \eqref{reduced}i)-iii) is obtained by imposing on the inviscid two-dimensional Boussinesq equations
\begin{equation}
\label{boussinesq}
\begin{cases}
u_t+(u\cdot\nabla)u=-\nabla p+\theta\,e_{2},
\\
\nabla\cdot u=0,
\\
\theta_t+u\cdot\nabla\theta=0
\end{cases}
\end{equation}
a \emph{stagnation-point similitude} velocity field on an infinitely long 2d channel $\Omega\equiv\{(x,y)\in[0,1]\times(0,+\infty)\}$. More particularly, due to incompressibility there exists a scalar stream function $\psi(t,x,y)$ such that $u=\nabla^{\perp}\psi=(\psi_y,-\psi_x)$. If we consider only stream functions of the form $\psi(t,x,y)=yf(t,x)$, then \eqref{reduced}i)-iii) arises from \eqref{boussinesq} with
\begin{equation}
\label{ansatz}
\begin{split}
u(t,x,y)=(f(t,x),-yf_x(t,x)),\qquad\,\quad\,\,\,\theta(t,x,y)=y\rho(t,x).
\end{split}
\end{equation}
In \eqref{boussinesq}, $u$ denotes the two-dimensional fluid velocity, $p$ the scalar pressure, $e_{2}$ the standard unit vector in the vertical direction, and $\theta$ represents either the temperature in the context of thermal convection, or the density in the modeling of geophysical fluids.

\vskip .1in

Note that periodicity \eqref{pbc}i), ii) of $f(t,x)$ results from periodicity (in $x$) of $u(t,x,y)$, i.e. $u(t,1,y)=u(t,0,y)$. For reasons that will be evident in \S\ref{uxbound}, whenever the periodic boundary condition \eqref{pbc} is under consideration, we will impose on the pressure $p(t,x,y)$ the boundary condition
\begin{equation}
\label{pressure}
p(t,1,y)=p(t,0,y)
\end{equation}
and assume $f(0,x)=f_0(x)$ satisfies the mean-zero condition
\begin{equation}
\label{meanzero}
\begin{split}
\int_0^1{f_0(x)\,dx}=0.
\end{split}
\end{equation}

The Boussinesq equations model large scale atmospheric and oceanic flows responsible for cold fronts and the jet stream (see e.g. \cite{gill} \cite{majda}). In addition, the Boussinesq equations also play an important role in the study of Rayleigh-Benard convection (see, e.g. \cite{drazin} \cite{constantin1}). Mathematically, the 2D Boussinesq equations serve as a lower-dimensional model of the 3D hydrodynamics equations and retain some key features, such as vortex stretching, of the 3D Euler equations. It is also well-known that (away from the axis of symmetry) the inviscid 2D Boussinesq equations are closely related to the Euler equations for 3D axisymmetric swirling flows (\cite{majdabertozzi}). The reader may refer to \cite{weinan} \cite{chae1} \cite{Wu}  for local existence results and blowup criteria for \eqref{boussinesq} and related models.

\vskip.1in

If $\theta\equiv0$, \eqref{boussinesq} reduces to the 2d incompressible Euler equations, while \eqref{reduced}i), ii) simplifies to
\begin{equation}
\label{reduced2d}
f_{xt}+ff_{xx}-f_x^2=-2\int_0^1{f_x^2\,dx}.
\end{equation}
Equation \eqref{reduced2d} is known as the inviscid Proudman-Johnson equation (\cite{proudman}). In \cite{Sarria}, a general solution formula for solutions of \eqref{reduced2d}, along with blowup and global-in-time criteria, were established (see \cite{Childress} \cite{Aconstantin} \cite{saxton} \cite{okamoto1} \cite{sarria1} for additional regularity results). Equation \eqref{reduced2d} is interesting in its own right from a mathematical perspective: it illustrates how the boundary conditions, more particularly periodic or Dirichlet boundary conditions, can either contribute to, or suppress, the formation of spontaneous singularities from smooth initial conditions in nonlinear evolution equations (\cite{Sarria}). Moreover, \eqref{reduced2d} appears as a reduced 1D model for the 3D inviscid primitive equations of large scale oceanic and atmospheric dynamics (\cite{cao1}), and is also related to the hydrostatic Euler equations (\cite{Wong} \cite{Kukavica}).

\vskip.1in

The term `stagnation-point similitude' arises from the observation that velocity fields of the form \eqref{ansatz}i) emerge from the modeling of flow near a stagnation point (\cite{stuart} \cite{ohkitani1} \cite{gibbon}). The study of solutions of the form \eqref{ansatz}i) appears to have started with Stuart (\cite{stuart2}); he considered solutions of the 3d incompressible Euler equations that had linear dependence in two variables $x$ and $z$, and showed that the resulting differential equations in the remaining independent variables $y$ and $t$ displayed finite time singular behavior. Since then, velocities of stagnation-point type have been used in the context of 3d Navier-Stokes and magneto-hydrodynamics equations (\cite{stuart2} \cite{constantin2} \cite{gibbon1} \cite{gibbon2}). Due to an infinite geometric structure in the $y$ direction, the velocity field \eqref{ansatz} possesses infinite energy when considered over the entire spatial domain $\Omega$; however, we believe that the analysis of reduced models such as \eqref{reduced} can provide valuable insights into the global regularity problem for the full 2d Boussinesq and the 3d axisymmetric Euler equations. For instance, recent numerical simulations (\cite{Luo}) indicate that solutions of the 3d axisymmetric Euler equations develop a singularity in finite time, precisely, at points where the velocity field has a stagnation point.

\vskip.1in

Below we summarize the main results of this paper.

\begin{theorem}
\label{local}
Consider the IBVP \eqref{reduced}-\eqref{dbc} (or \eqref{reduced} with \eqref{pbc} and \eqref{meanzero}). If\, $
f_0\in H^2([0,1])$,\, $f_0'\in L^\infty([0,1])$\, and\, $\rho_0\in H^1([0,1])$, then there
exists $T=T(\|f_0\|_{H^2}, \|f_0'\|_{L^\infty}, \|\rho_0\|_{H^1})>0$
such that \eqref{reduced} has a unique solution $(f,\rho)$
on $[0,T]$ satisfying
$$f\in C([0,T]; H^2),\qquad f_x\in C([0,T]; L^\infty),\qquad \rho\in C([0,T]; H^1).$$
Moreover, if
\begin{equation*}
\int_0^{T^*} \|f_x(t, \cdot)\|_{L^\infty} \,dt <+\infty,
\end{equation*}
then the local solution can be extended to $[0,T^*]$.
\end{theorem}

\begin{theorem}
\label{main}
Consider the IVP \eqref{reduced} with nontrivial smooth initial data $f_0(x)$ and $\rho_0(x)$ satisfying the Dirichlet boundary condition \eqref{dbc}. Suppose $\rho_0(x)\geq0$ for all $x\in[0,1]$ and denote by $x_i^*$, $0\leq i\leq n$, the finite number of points in $[0,1]$ where $f_0'(x)$ attains its greatest positive value. If the $x_i^*$ are located \textsl{only} at the boundary, and at each $x_i^*$ the initial vorticity satisfies $f_0''(x_i^*)\neq0$, then there exists a finite $t^*>0$ such that
\begin{equation*}
\lim_{t\nearrow t^*}\int_0^t{f_x(s,x_i^*)\,ds}=+\infty,\qquad\qquad \lim_{t\nearrow t^*}|f_{xx}(t,x_i^*)|=+\infty.
\end{equation*}
In contrast, if $x_i^*\in[0,1]$, then there exist nontrivial $f_0(x)$ and $\rho_0(x)\geq0$ satisfying Dirichlet boundary condition \eqref{dbc}, or periodic boundary condition \eqref{pbc} with mean-zero \eqref{meanzero}, such that if the initial vorticity $f_0''(x)$ vanishes at $x_i^*$ for at least one $i$, then the corresponding solution of \eqref{reduced} persists for all time.
\end{theorem}

The outline for the remainder of the paper is as follows. In \S\ref{uxbound}, the local well-posedness of \eqref{reduced}-\eqref{dbc} (and \eqref{reduced} with \eqref{pbc} and \eqref{meanzero}) is established along with a regularity criterion in terms of the time integral of  $\left\|f_x(t,\cdot)\right\|_{L^\infty([0,1])}$. In \S \ref{boundary}, we prove the existence of general, nontrivial smooth initial conditions, satisfying Dirichlet boundary conditions \eqref{dbc}, for which the time integral of $f_x(t,x)$ blows up in finite time at the boundary. Moreover, we also show that this blowup implies either one-sided or two-sided blowup in the vorticity\footnote[1]{By two-sided blowup we mean simultaneous blowup to both positive and negative infinity.}. Our blowup criteria is local-in-space and relies both on initial velocities with a local profile characterized by the non-vanishing of $f_0''(x)$ at the boundary and non-negativity of the initial temperature $\rho_0(x)$. Due to the local nature of the blowup criteria, our results do not rule out the formation of finite-time singularities either in the interior of the domain or at the boundary if $f_0$ possesses a different local structure. Thus, in \S \ref{global} we follow an argument similar to that in \cite{Childress} to construct a family of global solutions of \eqref{reduced} which provides valuable insights on the type of initial conditions needed to suppress finite-time blowup. The reader may then refer to \S \ref{conclusions} for concluding remarks.

\section{Local Well-posedness and Regularity Criteria}
\label{uxbound}

This section presents a regularity criterion which,
together with Theorem
\ref{blowup1} of \S\ref{boundary}, states that a finite time singularity of \eqref{reduced}-\eqref{dbc} (or \eqref{reduced} with \eqref{pbc} and \eqref{meanzero}) develops if and
only if the time integral of $f_x$ becomes infinity in a finite time. In addition, the local well-posedness of both boundary value problems is also presented.

\begin{theorem}
\label{Criterion}
Consider the IVP \eqref{reduced}.
Assume $f_0$ and $\rho_0$ satisfy either the Dirichlet boundary condition (\ref{dbc}), or the periodic boundary condition (\ref{pbc}) with mean-zero condition \eqref{meanzero}, and suppose
$$
f_0\in H^2([0,1]), \quad f_0'\in L^\infty([0,1]), \quad \rho_0\in H^1([0,1]).
$$
Then there
exists $T=T(\|f_0\|_{H^2}, \|f_0'\|_{L^\infty}, \|\rho_0\|_{H^1})>0$
such that \eqref{reduced} has a unique solution $(f,\rho)$
on $[0,T]$ satisfying $f\in C([0,T]; H^2)$,
$f_x\in C([0,T]; L^\infty)$ and $\rho\in C([0,T]; H^1)$.
Moreover, if
\begin{equation} \label{Cricon}
\int_0^{T^*} \|f_x(t, \cdot)\|_{L^\infty} \,dt <+\infty,
\end{equation}
then the local solution can be extended to $[0,T^*]$.
\end{theorem}

Recall that the global regularity problem for the 2d inviscid Boussinesq equations \eqref{boussinesq} with arbitrary `smooth enough' initial data is currently open. Local solutions can be extended into global ones if
either one of the criteria,
$$
\int_0^{\infty}{\left\|\nabla u\right\|_{\infty}dt}<+\infty\qquad\text{or}\qquad \int_0^{\infty}{\left\|\nabla \theta\right\|_{\infty}dt}<+\infty
$$
holds. The criterion in Theorem \ref{Criterion} reflects the criterion
in terms of the velocity field $u$ for the 2d Boussinesq equations. There is no criterion corresponding
to the one on $\theta$ for \eqref{reduced}-\eqref{dbc}, namely no criterion in terms of $\rho$. The main reason is that \eqref{reduced}-\eqref{dbc} could still blow up in a finite time even if
$\rho \equiv 0$.

\vskip .1in

Before proving Theorem \ref{Criterion}, note that in the periodic case, the pressure boundary condition \eqref{pressure} and the mean-zero assumption \eqref{meanzero} imply that
\begin{equation}
\label{meanzerof}
\begin{split}
\int_0^1{f(t,x)\,dx}\equiv0
\end{split}
\end{equation}
for as long as $f$ is defined. This is a consequence of integrating the horizontal component of \eqref{boussinesq}i), which for solutions of the form \eqref{ansatz} reduces to
\begin{equation}
\label{horizontal}
\begin{split}
f_t+ff_x=-p_x.
\end{split}
\end{equation}

We now state and prove the following elementary lemma.

\begin{lemma}\label{ele}
Assume $f$ satisfies the Dirichlet boundary condition (\ref{dbc}), or the periodic boundary condition (\ref{pbc}) with mean-zero condition \eqref{meanzero}.
Suppose $f_x\in L^2([0,1])$. Then, for a constant $C$,
$$
\|f\|_{L^\infty([0,1])} \le C\,\|f'\|_{L^2([0,1])}.
$$
In particular, $\|f\|_{L^2([0,1])} \le C\,\|f'\|_{L^2([0,1])}$.
\end{lemma}
\begin{proof} The proof is simple. In the case
of the Dirichlet boundary condition,
$$
|f(x)| = \left|\int_0^x f'(y)\,dy \right| \le \|f'\|_{L^2([0,1])}.
$$
In the case of the periodic boundary condition, we write
$$
f(x) =\sum_{k} \widehat{f}(k)\,e^{ix k}, \qquad
\widehat{f}(k) =\int_0^1 e^{-ikx}\, f(x)\,dx.
$$
Thus, using \eqref{meanzerof}, we obtain

$$
\|f\|_{L^\infty} \le C\, \left[\sum_{k\not =0}|k|^2 \,|\widehat{f}(k)|^2\right]^{1/2} = C\, \|f'\|_{L^2}.
$$
This proves Lemma \ref{ele}.\hfill $\square$
\end{proof}

\vskip .1in

\begin{proof}
The local well-posedness can be obtained through an approximation procedure
(see, e.g, \cite{majdabertozzi}). For the sake of brevity, we shall just
provide the key component of this procedure, namely the local bound for $\|f\|_{H^2} + \|\rho\|_{H^1}$. In order to establish the desired local
bound, we consider the norm
\begin{equation} \label{Yt}
Y^2(t) \equiv \|\rho(t,\cdot)\|_{H^1}^2 + \|f_x(t,\cdot)\|^2_{L^2}
+ \|f_x(t,\cdot)\|_{L^\infty} + \|f_{xx}(t,\cdot)\|^2_{L^2}
\end{equation}
and show that
\begin{equation} \label{Yb}
Y^2(t) \le Y^2(0) + C\, \int_0^t (Y^2(\tau) + Y^3(\tau) + Y^4(\tau))\,d\tau.
\end{equation}
Gronwall's inequality then implies that, for some $T=T(Y(0))>0$ and
$t\in [0,T]$,
$$
Y(t) <\infty.
$$
This also gives a local bound for $\|f\|_{L^2}$ due to
Lemma \ref{ele}. We remark that $\|f_x(t,\cdot)\|_{L^\infty}$
is included in $Y$ because it appears to be more convenient to obtain
a ``closed" differential inequality by considering this norm
simultaneously. We now prove (\ref{Yb}) through energy estimates. Taking the inner product of \eqref{reduced}ii)  with $\rho$ and integrating by parts, we have
\begin{equation}
\label{ineq01}
\frac{d}{dt}\int_0^1{\rho^2\,dx}=3\int_0^1{\rho^2 f_x\,dx}
\leq
3\,\|f_x\|_{L^\infty}\, \int_0^1{\rho^2\,dx}.
\end{equation}
Taking $\partial_x$ of \eqref{reduced}ii), dotting with $\partial_x\,\rho$, integrating by parts and applying Lemma \ref{ele}, we obtain
\begin{equation}
\label{ineq02}
\begin{split}
\frac{d}{dt}\int_0^1{\rho_{x}^2\,dx}
=\int_0^1{f_x\,\rho_x^2\,dx}+2\int_0^1{\rho\rho_xf_{xx}\,dx}
&\leq \|f_x\|_{L^\infty}\,\int_0^1{\rho_x^2\,dx}+
\|\rho\|_{L^\infty}\,\int_0^1({\rho_x^2 + f_{xx}^2)\,dx}
\\
&\leq \|f_x\|_{L^\infty}\,\|\rho_x\|_{L^2}^2 + C\,\|\rho_x\|_{L^2}\,
(\|\rho_x\|_{L^2}^2 + \|f_{xx}\|_{L^2}^2).
\end{split}
\end{equation}
We remark that, in the case of periodic boundary conditions, we use 
$$
\|\rho\|_{L^\infty} \le C (\|\rho\|_{L^2} + \|\rho_x\|_{L^2})
$$
instead of Lemma \ref{ele} to avoid the mean-zero assumption on $\rho$. 
This inequality holds without $\rho$ being mean-zero in the periodic case. 
Dotting \eqref{reduced}i) with $f_x$ and using \eqref{dbc} or \eqref{pbc}, we find
\begin{equation}
\label{ineq03}
\begin{split}
\frac{d}{dt}\int_0^1{f_{x}^2\,dx}= 3 \int_0^1f_{x}^3\,dx -2\,\int_0^1 \rho\, f_x\,dx
\le 3 \, \|f_x\|_{L^\infty}\, \|f_x\|_{L^2}^2
+ \|\rho\|_{L^2}^2 + \|f_x\|_{L^2}^2.
\end{split}
\end{equation}
Similarly,
\begin{equation}
\label{ineq04}
\frac{d}{dt}\int_0^1{f_{xx}^2\,dx} \le  (3 \|f_x\|_{L^\infty}\,
+ 1)\int_0^1f_{xx}^2\,dx +  \|\rho\|_{L^2}^2.
\end{equation}
Now define the Lagrangian path $\gamma(t,x)$ via the initial value problem
\begin{equation}
\label{charac}
\begin{split}
\dot\gamma(t,x)=f(t,\gamma(t,x)),\qquad\qquad\gamma(0,x)=x,
\end{split}
\end{equation}
where $\cdot\equiv\frac{d}{dt}$. Invoking (\ref{charac}) in \eqref{reduced}i), taking the $L^\infty$-norm
and using Lemma \ref{ele}, we have
\begin{equation}
\label{ineq05}
\begin{split}
\|f_x(t,\cdot)\|_{L^\infty} &\le \|f'_{0}\|_{L^\infty} + \int_0^t (\|\rho\|_{L^\infty} + \|f_x\|^2_{L^\infty} + I(\tau))\,d\tau
\\
&\le \|f'_{0}\|_{L^\infty} + \int_0^t (\|\rho_x\|_{L^2}^2 + \|f_x\|^2_{L^\infty} + \|\rho\|_{L^2} + 2 \|f_x\|_{L^2}^2)\,d\tau.
\end{split}
\end{equation}
It is then easy to see that combining (\ref{ineq01}) through
(\ref{ineq05}) yields the desired inequality in (\ref{Yb}). This
completes the local well-posedness part. To prove the regularity criterion,  it suffices to show that (\ref{Cricon}) implies
the bound
\begin{equation}\label{fH2}
f \in L^\infty([0, T^*]; H^2), \quad
f_x \in L^\infty([0, T^*]; L^\infty)\quad\text{and} \quad \rho\in L^\infty([0, T^*]; H^1).
\end{equation}
Adding the inequalities in (\ref{ineq01}) through
(\ref{ineq04}) yields
\begin{equation}
\label{ineq06}
\begin{split}
\frac{d}{dt}\int_0^1 (\rho^2 + \rho_x^2 + f_x^2 + f_{xx}^2)\,dx
&\le C\,(1+ \|f_x\|_{L^\infty})\,\int_0^1 (\rho^2 + \rho_x^2 + f_x^2 + f_{xx}^2)\,dx
\\
&\quad  + \|\rho\|_{L^\infty}\,\int_0^1({\rho_x^2 + f_{xx}^2)\,dx}.
\end{split}
\end{equation}
Invoking (\ref{charac}) in \eqref{reduced}ii) and
taking the $L^\infty$-norm, we have
\begin{equation} \label{rhoin}
\|\rho(t, \cdot)\|_{L^\infty}\, \le \|\rho_0\|_{L^\infty}\,e^{\int_0^t \|f_x\|_{L^\infty}\,d\tau}.
\end{equation}
Combining (\ref{Cricon}), (\ref{ineq06}) and (\ref{rhoin}) leads to
$$
f_x, f_{xx} \in L^\infty([0, T^*]; L^2) \quad\text{and} \quad \rho\in L^\infty([0, T^*]; H^1).
$$
Lemma \ref{ele} also yields $f\in L^\infty([0, T^*]; L^2)$. Furthermore, applying Gronwall's inequality to (\ref{ineq05}) leads to
$$
f_x \in L^\infty([0, T^*]; L^\infty).
$$
This establishes (\ref{fH2}). We have thus completed the proof of Theorem
\ref{Criterion}.\hfill $\square$
\end{proof}

\section{Blowup}
\label{boundary}

In this section we prove the existence of solutions to \eqref{reduced}, satisfying Dirichlet boundary conditions \eqref{dbc}, which blowup in finite time from nontrivial smooth initial data. Our blowup criteria is in terms of an arbitrary nonnegative initial temperature $\rho_0$ and the local profile of a nontrivial initial velocity $f_0$ near the boundary. More particularly, note that the vorticity associated to the velocity field \eqref{ansatz} is given, after a slight abuse of notation, by
\begin{equation}
\label{vorticity}
\nabla\times u=-yf_{xx}(t,x),
\end{equation}
so that we may refer to $f_0''(x)$ as the initial vorticity. We examine how the global regularity of solutions of \eqref{reduced} is affected by both, the corresponding boundary condition and the (non)vanishing of the initial vorticity at points where $f_0'(x)$ attains its maximum. Briefly, using \eqref{reduced}ii) and \eqref{charac}, we first write \eqref{reduced}i) as a linear second-order, non-homogeneous ode in terms of $\gamma_x^{-1}$. Then,  a ``conservation in mean'' condition for $\gamma_x$ will allow us to solve this differential equation and obtain an implicitly defined representation formula for $\gamma_{x}$. The blowup is then established by deriving lower bounds on $\gamma_x$ which depend on the profile of $f_0$ near the boundary. Lastly, using a representation formula for $f_{xx}(t,\gamma(t,x))$ in terms of $\gamma_x$, we prove blowup of the vorticity \eqref{vorticity}. We begin by establishing some preliminary results.

\vskip.1in

Note that the classical existence and uniqueness result for odes (as applied to the IVP \eqref{charac}), along with Dirichlet or periodic boundary conditions, implies that
\begin{equation}
\label{fixed}
\gamma(t,0)\equiv0,\qquad\quad\gamma(t,1)\equiv1
\end{equation}
or respectively
\begin{equation}
\label{fixedpbc}
\gamma(t,x+1)-\gamma(t,x)\equiv1,
\end{equation}
for as long as a solution exists. In either case, the mean of $\gamma_x$ over $[0,1]$ is preserved in time:
\begin{equation}
\label{meanone}
\int_0^1{\gamma_{x}\,dx}\equiv1.
\end{equation}
Now, differentiating \eqref{charac} with respect to $x$ yields
\begin{equation}
\label{jacobian0}
\dot\gamma_{x}=f_x(t,\gamma(t,x))\,\gamma_{x}\,,
\end{equation}
which we integrate to obtain
\begin{equation}
\label{uxjac}
\gamma_{x}(t,x)=\text{exp}\left(\int_0^t{f_x(s,\gamma(s,x))\,ds}\right).
\end{equation}
But using \eqref{charac}i) and \eqref{uxjac} on equation \eqref{reduced}ii), we find that
\begin{equation}
\label{rho}
\begin{split}
\rho(t,\gamma(t,x))&=\rho_0(x)\,\gamma_{x}(t,x).
\end{split}
\end{equation}
Then differentiating \eqref{jacobian0} with respect to time and using (\ref{reduced})i) and \eqref{rho}, yields
\begin{equation}
\label{der2}
I(t)-\rho_0\,\gamma_{x}=-\gamma_{x}\,\left( \gamma_{x}^{-1}\right)^{\ddot{}}.
\end{equation}

Setting $\omega=\gamma_x^{-1}$ in \eqref{der2} now gives
\begin{equation}
\label{ode0}
\ddot\omega(t,x)+I(t)\omega(t,x)=\rho_0(x),
\end{equation}
a second-order linear, non-homogeneous ode parametrized by $x\in[0,1]$ and complemented by the initial values $ \omega(0,x)\equiv1$ and $\dot\omega(0,x)=-f_0'(x)$. We use variation of parameters to write down the form of its the general solution. 

\vskip .1in

First consider the associated homogeneous equation
\begin{equation}
\label{homo1}
\ddot\omega_h(t,x)+I(t)\omega_h(t,x)=0.
\end{equation}
Let $\phi_1(t)$ and $\phi_2(t)$ be two linearly independent solutions of \eqref{homo1} satisfying $\phi_1(0)=\dot\phi_2(t)=1$ and $\dot\phi_1(0)=\phi_2(0)=0$. Setting $\phi_2(t)=\eta(t)\phi_1(t)$ we obtain, via reduction of order, the general solution of \eqref{homo1} as
\begin{equation}
\label{omegah}
\begin{split}
\omega_h(t,x)
=c_1(x)\phi_1(t)+c_2(x)\phi_2(t)
=\phi_1(t)(c_1(x)+c_2(x)\eta(t)),
\end{split}
\end{equation}
where the strictly increasing function $\eta(t)$ satisfies
\begin{equation}
\label{eta}
\dot\eta(t)=\phi_1(t)^{-2},\qquad\quad\qquad\eta(0)=0.
\end{equation}
Next, following a standard variation of parameters argument, we look for a particular solution to \eqref{ode0} of the form
\begin{equation}
\label{omegap}
\omega_p(t,x)=v_1(t,x)\phi_1(t)+v_2(t,x)\phi_2(t),
\end{equation}
where $v_1$ and $v_2$ are to be determined. This yields
$$v_1(t,x)=a(x)-\rho_0(x)\int_0^t{\eta(s)\phi_1(s)\,ds},\qquad
v_2(t,x)=b(x)+\rho_0(x)\int_0^t{\phi_1(s)\,ds}$$
for arbitrary functions $a(x)$ and $b(x)$. The general solution of \eqref{ode0}i), $\omega=\omega_h+\omega_p$, now becomes
\begin{equation}
\label{gen0}
\omega(t,x)=\phi_1(t)\left[1-f_0'(x)\eta(t)-\rho_0(x)\left(\int_0^t{\eta(s)\phi_1(s)\,ds}-\eta(t)\int_0^t{\phi_1(s)\,ds}\right)\right],
\end{equation}
where we used the initial values for $\omega$, along with $\eta(0)=0$ and $\dot\eta(0)=1$, to obtain $c_1(x)+a(x)\equiv1$ and $c_2(x)+b(x)=-f_0'(x)$. Lastly, since $\gamma_x=\omega^{-1}$, the conservation of mean \eqref{meanone} and formula \eqref{gen0} imply that $\phi_1(t)$ satisfies the relation
\begin{equation}
\label{phi}
\phi_1(t)=\int_0^1{\left(1-\eta(t)f_0'(x)-\rho_0(x)g(t)\right)^{-1}dx},\quad\quad g(t)=\int_0^t{\eta(s)\phi_1(s)\,ds}-\eta(t)\int_0^t{\phi_1(s)\,ds},
\end{equation}
which yields the implicitly defined representation formula
\begin{equation}
\label{jacobian}
\gamma_{x}(t,x)=\left[\phi_1(t)\left({1-\eta(t)f_0'(x)-\rho_0(x)g(t)}\right)\right]^{-1}.
\end{equation}

Before proving Theorem \eqref{main}, we make the following observation.

\vskip .1in
Define the positive real number $\eta_*$ by
\begin{equation}
\label{eta*}
\eta_*=\frac{1}{M_0}\qquad\text{for}\qquad M_0\equiv\max_{x\in[0,1]}f_0'(x).
\end{equation}

\begin{lemma}
\label{phipos}
If\, $0\leq \eta<\eta_*$\, on\, $\Sigma\equiv[0,T)$\, for some\, $0<T\leq+\infty$,\, then\, $\phi_1(t)>0$\, on\, $\Sigma$. Additionally, if $\rho_0(x)\geq0$\, for all\, $x\in[0,1]$,\, then\, $0<\phi_1(t)<+\infty$\, for all\, $t\in\Sigma$.
\end{lemma}
\begin{proof}
Let $0<T\leq+\infty$ be such that $\eta$, with $\eta(0)=0$ and $\dot\eta(0)=1$, satisfies $0\leq \eta<\eta_*$ for all $t\in\Sigma\equiv[0,T)$. The first part of the Lemma follows directly from the boundedness of $\eta$ on $\Sigma$, the IVP \eqref{eta}, and $\phi_1(0)=1$. Now, in addition to the above, suppose $\rho_0(x)\geq0$ for all $x\in[0,1]$, and assume there is $t_1\in\Sigma$ such that
\begin{equation}
\label{blowphi}
\lim_{t\nearrow t_1}\phi_1(t)=+\infty.
\end{equation}
Since $\phi_1>0$ on $\Sigma$, then
$$\dot g(t)=-\dot\eta(t)\int_0^t{\phi_1(s)\,ds}=-\phi_1(t)^{-2}\int_0^t{\phi_1(s)\,ds}<0$$
for all $t\in\Sigma$. This, along with $g(0)=0$ and $\rho_0(x)\geq0$, implies that, on $\Sigma$,
\begin{equation}
\label{posden}
1-\eta(t)f_0'(x)-\rho_0(x)g(t)\geq1-\eta(t)f_0'(x)>0
\end{equation}
for all $x\in[0,1]$. Consequently, \eqref{phi}i) yields
\begin{equation}
\label{phiineq}
\phi_1(t)^{-1}\geq\left(\int_0^1{\frac{dx}{1-\eta(t)f_0'(x)}}\right)^{-1}>0,\qquad\text{for all $t\in\Sigma$}.
\end{equation}
But using \eqref{blowphi} on \eqref{phiineq} we obtain
$$\lim_{t\nearrow t_1}\int_0^1{\frac{dx}{1-\eta(t)f_0'(x)}}=+\infty,\qquad t_1\in\Sigma,$$
and so $\lim_{t\nearrow t_1}\eta(t)=\eta_*$, contradicting our assumption that $0\leq\eta<\eta_*$ for all $t\in\Sigma$.\hfill $\square$
\end{proof}

We now establish the following blowup result.

\begin{theorem}
\label{blowup1}
Consider the IVP \eqref{reduced} for smooth nontrivial initial data $f_0(x)$ and $\rho_0(x)$ satisfying the Dirichlet boundary condition \eqref{dbc}. Suppose $\rho_0(x)\geq0$ for all $x\in[0,1]$ and assume $f_0'(x)$ attains its greatest value $M_0>0$ \emph{only} at boundary point(s) $x_i^*\in\{0,1\}$,\, $i=0,1$. If the initial vorticity $f_0''(x)$ is non-zero at each $x_i^*$, then there exists a finite time $t^*>0$ such that
\begin{equation}
\label{blowup}
\lim_{t\nearrow t^*}\int_0^t{f_x(s,x_i^*)\,ds}=+\infty.
\end{equation}
\end{theorem}
\begin{proof}
Suppose $0\leq\eta<\eta_*=1/M_0$ for all $t\in\Sigma=[0,t^*)$ and some $0<t^*\leq+\infty$. For simplicity, assume $f_0'(x)$ attains its largest value $M_0$ only at $x^*=0$ with $f_0''(0)\neq0$. Further, suppose $\rho_0(x)\geq0$ for all $x\in[0,1]$. First we show that $\gamma_x(t,0)\to+\infty$ as $\eta\nearrow \eta_*$. Then we prove that as $\eta$ approaches $\eta_*$, $t$ approaches a finite time $t^*>0$.

\vskip.1in

For all $t\in\Sigma$ and $x\in[0,1]$, \eqref{jacobian}, \eqref{posden} and \eqref{phiineq} imply that
\begin{equation}
\label{jacineq}
\gamma_x(t,x)\geq\left(\int_0^1{\frac{dx}{1-\eta(t)f_0'(x)}}\right)^{-1}\left(\frac{1}{1-\eta(t)f_0'(x)-\rho_0(x)g(t)}\right)>0,
\end{equation}
so that
\begin{equation}
\label{jacineq2}
\gamma_x(t,0)\geq\left(\int_0^1{\frac{dx}{1-\eta(t)f_0'(x)}}\right)^{-1}\left(\frac{1}{1-\eta(t)M_0}\right)
\end{equation}
for all $t\in\Sigma$. We need to estimate the integral term in \eqref{jacineq2}. Smoothness of $f_0$ implies, via a Taylor expansion about $x=0$, that
\begin{equation}
\label{eq:againdiri}
\begin{split}
\epsilon+M_0-f_0'(x)\sim\epsilon+\left|C_1\right|x
\end{split}
\end{equation}
for $0\leq x\leq r\leq1$,\, $C_1=f_0''(0)<0$\, and some $\epsilon>0$. In \eqref{eq:againdiri} we use the notation
\begin{equation}
\label{eq:simexplanation}
h(x) \sim L + w(x),
\end{equation}
valid for $0\leq|x-\beta|\leq s$, to mean that there exists a function $v(x)$ defined on $(\beta-r,\beta+r)$ such that
\begin{equation}
\label{eq:simexplanation2}
h(x)-L=w(x)(1+v(x))\,\,\,\,\,\,\,\,\,\,\,\,\text{where}\,\,\,\,\,\,\,\,\,\,\,\,\lim_{x\rightarrow\beta}v(x)=0.
\end{equation}
Using \eqref{eq:againdiri} we obtain the estimate
\begin{equation}
\label{eq:app}
\begin{split}
\int_{0}^{r}{\frac{dx}{\epsilon+M_0-f_0'(x)}}\sim\int_{0}^{r}{\frac{dx}{\epsilon+\left|C_1\right|x}}=-\frac{1}{\left|C_1\right|}\ln\epsilon
\end{split}
\end{equation}
for $\epsilon>0$ small. If we now set $\epsilon=\frac{1}{\eta}-M_0$ into (\ref{eq:app}), we see that for $\eta_*-\eta>0$ small,
\begin{equation}
\label{eq:intest1}
\begin{split}
\int_0^1{\frac{dx}{1-\eta(t)f_0'(x)}}\sim-\frac{M_0}{\left|C_1\right|}\ln(\eta_*-\eta),
\end{split}
\end{equation}
which we use on \eqref{jacineq2} to obtain
\begin{equation}
\label{jacineq3}
\gamma_x(t,0)\geq\left(\int_0^1{\frac{dx}{1-\eta(t)f_0'(x)}}\right)^{-1}\left(\frac{1}{1-\eta(t)M_0}\right)\sim -\frac{C}{(\eta_*-\eta)\ln(\eta_*-\eta)}
\end{equation}
for $C$ a positive constant. The above implies that
$$\gamma_x(t,0)\to+\infty\qquad\quad\text{as}\quad\qquad \eta\nearrow\eta_*.$$
Last we establish the existence of a finite blowup time
\begin{equation}
\label{t*gen}
t^*\equiv\lim_{\eta\nearrow\eta_*}t(\eta)>0.
\end{equation}
For $\eta_*-\eta>0$ small, \eqref{eta}, \eqref{phiineq} and \eqref{eq:intest1} yield
\begin{equation}
\label{time0}
0<\frac{dt}{d\eta}\leq\left(\int_0^1{\frac{dx}{1-\eta(t)f_0'(x)}}\right)^{2}\sim C\ln^2(\eta_*-\eta).
\end{equation}
Consequently,
\begin{equation}
\label{time1}
\begin{split}
0<t^*-t&\leq(\eta_*-\eta)\left[1+\left(\ln(\eta_*-\eta)-1\right)^2\right],
\end{split}
\end{equation}
the right-hand side of which vanishes as $\eta\nearrow \eta_*$. In fact, using \eqref{eta}, \eqref{phiineq} and Lemma \ref{phipos}, it follows that
\begin{equation}
\label{t}
\begin{split}
t(\eta)\leq\int_0^{\eta}{\left(\int_0^1{\frac{dx}{1-\mu f_0'(x)}}\right)^2d\mu}
\end{split}
\end{equation}
for $0\leq\eta<\eta_*$. Inequality \eqref{time1} then implies that the integral in \eqref{t} remains finite as $\eta\nearrow\eta_*$ and, further, that an upper-bound for the blowup time \eqref{t*gen} is
\begin{equation}
\label{upperbound}
\begin{split}
0<t^*\leq\lim_{\eta\nearrow\eta_*}\int_0^{\eta}{\left(\int_0^1{\frac{dx}{1-\mu f_0'(x)}}\right)^2d\mu}.
\end{split}
\end{equation}
\end{proof}
\hfill $\square$


\begin{remark}
\label{noblow}
A simple choice of initial data to which the blowup result in Theorem \ref{blowup1} applies is $f_0(x)=x(1-x)$ and $\rho_0(x)=\sin^2(2\pi x)$. In this case \eqref{upperbound} yields $\pi^2/6\sim 1.65$ as an upper-bound for the blowup time of $\gamma_x$ at $x^*=0$. Clearly, this choice of $f_0(x)$ does not satisfy the periodic boundary conditions \eqref{pbc}, but if instead we choose the mean-zero function $f_0(x)=\sin(2\pi x)$ and the same $\rho_0$ as above, then for $x_i^*=0, 1$,\, we have that $\gamma_x(t,x^*_i)\to+\infty$\,
no slower than $(\eta_*-\eta)^{-1/2}$ as $\eta\nearrow\eta_*=1/(2\pi)$. However, for this choice of $f_0$, \eqref{time1} now becomes
\begin{equation}
\label{noupperbound}
\begin{split}
0<t^*-t\leq-\ln(\eta_*-\mu)\big|_{\eta}^{\eta_*}=+\infty.
\end{split}
\end{equation}
Thus, for the latter choice of initial data we fail to establish a finite upper-bound for the blowup time. As opposed to the case $f_0(x)=x(1-x)$, in which finite-time blowup occurs, we remark that \eqref{noupperbound} is a result of $x_i^*=0, 1$ now being inflection points of $f_0(x)=\sin(2\pi x)$. A similar result follows when at least one of the $x_i^*$ is an inflection point of $f_0$. In \S\ref{global} we elaborate on the above and discuss the effects that an initial vorticity which vanishes at the point(s) $x_i^*$ may have on the regularity of solutions of \eqref{reduced}.
\end{remark}

\begin{remark}
Since $f_0''(x_i^*)\neq 0$ is required for finite-time blowup, the assumption that $f_0'$ attains its greatest value $M_0$ only at boundary point(s) $x_i^*$ is needed for $f_0$ to be smooth; otherwise, if $x_i^*\in(0,1)$, then $f_0''(x_i^*)\neq0$ will imply a jump-discontinuity of finite magnitude in $f_0''(x)$ through $x_i^*$. Regularity criteria for non-smooth initial velocities, including piecewise-linear functions and maps with ``cusps'' and/or ``kinks'' on their graphs, can be studied via an argument similar to that used in the proof of Theorem \ref{blowup1} (see, e.g., \cite{Sarria} \cite{sarria1}).
\end{remark}


Lastly, we establish finite-time blowup of the vorticity \eqref{vorticity} under the setting of Theorem \ref{blowup1}.

\begin{corollary}
\label{blowup2}
Suppose the assumptions in Theorem \eqref{blowup1} hold. Then there exists a finite time $t^*>0$ such that the vorticity \eqref{vorticity} blows up as $t\nearrow t^*$. Further, if $f_0'(x)$ attains its maximum at both endpoints, then this blowup is two-sided.
\end{corollary}
\begin{proof}
Differentiating \eqref{jacobian} with respect to time and using \eqref{jacobian0} yields
\begin{equation}
\label{nasty1}
f_x(t,\gamma(x,t))=\phi_1(t)^{-2}\left(\frac{f_0'(x)-\rho_0(x)\int_0^t{\phi_1 ds}}{1-\eta(t)f_0'(x)-\rho_0(x)g(t)}\right)-\frac{\dot\phi_1}{\phi_1}.
\end{equation}
If we now differentiate the above in space and use \eqref{jacobian} we find that
\begin{equation}
\label{nasty2}
f_{xx}(t,\gamma(t,x))=h(t,x)\,\gamma_x
\end{equation}
for
\begin{equation}
\label{f}
h=f_0''-\rho_0'\int_0^t{\phi_1\, ds}+\left(\,\rho_0'f_0'-\rho_0f_0''\,\right)\int_0^t{\eta\phi_1\,ds}.
\end{equation}
Without loss of generality, assume $f_0'(x)$ achieves its maximum $M_0$ at both endpoints $x_0^*=0$ and $x_1^*=1$. Then setting $x=x_i^*$, $i=0, 1$, in \eqref{nasty2}-\eqref{f} and using \eqref{fixed}, gives
\begin{equation}
\label{fxxeq0}
f_{xx}(t,x_i^*)=\left(f_0''(x_i^*)+M_0\,\rho_0'(x_i^*)g^*(t)\right)\gamma_x(t,x_i^*)
\end{equation}
with
$$g^*(t)=\int_0^t{\eta(s)\phi_1(s)\,ds}-\eta_*\int_0^t{\phi_1(s)\,ds}.$$
Suppose $0\leq\eta<\eta_*$. Then by Lemma \ref{phipos},
\begin{equation}
\label{negg*}
g^*(t)\leq g(t)<0.
\end{equation}
Now, since $\rho_0(x)\nequiv 0$ is nonnegative and vanishes at the endpoints, then  $\rho_0'(0)\geq0$ and $\rho_0'(1)\leq0$. Moreover, since $M_0>0$ is the largest value attained by $f_0'(x)$ and $f_0''(x_i^*)\neq0$, then $f_0''(0)<0$, while $f_0''(1)>0$. Consequently, using \eqref{negg*} we set $i=0$ and respectively $i=1$ in \eqref{fxxeq0} to find
\begin{equation}
\label{twosided0}
f_{xx}(t,0)\leq f_0''(0)\,\gamma_x(t,0),\qquad\qquad
f_{xx}(t,1)\geq f_0''(1)\,\gamma_x(t,1).
\end{equation}
By letting $t$ approach the finite time $t^*>0$ established in Theorem \ref{blowup1}, we conclude that
\begin{equation}
\label{twosided}
f_{xx}(t,0)\to-\infty\qquad\quad\text{and}\qquad\quad
f_{xx}(t,1)\to+\infty.
\end{equation}
\end{proof}
\hfill $\square$

\begin{remark}
The issue of solutions of hydrodynamical-related models diverging at every point in their spatial domain and/or in only one direction of infinity has been studied previously (see e.g. \cite{grundy} \cite{constantin2} \cite{okamoto1} \cite{sarria2}). In the case where $M_0$ is attained at both boundary points (so that the two-sided blowup in \eqref{twosided} takes place), Corollary \ref{blowup2} gives conditions on the initial data which imply the existence of solutions of \eqref{reduced} whose slopes cannot blowup only towards one direction of infinity at every point in their domain.
\end{remark}


%
%

\section{An Infinite Family of Exact Global Solutions Spanning from Zero Initial Velocities}
\label{global}

The question of finite-time blowup in \eqref{reduced} from nontrivial initial velocities having a local profile different from that described in Theorem \ref{blowup1} is still open (see Remark \ref{noblow}). To help clarify this issue, in this Section we use an argument similar to that in \cite{Childress} to construct a family of global solutions to \eqref{reduced}. Our findings indicate that an initial nontrivial vorticity which vanishes at, at least, one of the $x_i^*$ (the points where $f_0'$ attains its maximum) is a necessary condition to arrest finite-time blowup. This, in turn, would imply that a boundary-induced singularity, possible only under the set-up of Theorem \ref{blowup1}, is the correct underlying mechanism for solutions of \eqref{reduced} to blowup from nontrivial smooth $f_0$.

\vskip .05in

For a constant $N_0\in\mathbb{R}^+\cup\{0\}$, we will consider initial data\, $\rho_0(x)=\sin^2\left(2\pi x\right)$\, and\, $f_0'(x)=-N_0\cos(4\pi x)$. Note that for $N_0>0$, $f_0'$ attains its greatest, positive value at points $x_i^*$ located in the interior, with all the $x_i^*$ being inflection points of $f_0$. As opposed to the finite-time blowup in Theorem \ref{blowup1}, we will find that solutions corresponding to this choice of initial data persist for all time. This leads us to conclude that the vanishing of the initial vorticity $f_0''(x)$ at $x_i^*$ is responsible for suppressing the blowup. Briefly, the family of solutions we construct features exponential decay of $\rho$ to zero as time goes to infinity, while $f_x$ convergences to steady states. The latter implies that both the velocity and the vorticity are uniformly bounded in time. Further, $\gamma_x$ grows exponentially at a finite number of points in $[0,1]$ but decays, also exponentially, everywhere else\footnote[2]{But the locations where it grows exponentially coincide with the points where $\rho_0(x)$ vanishes, which is the reason why $\rho$ only decays.}. So even though the solutions we construct persist for all time, the exponential growth of $\gamma_x$ at a finite number of locations and exponential decay everywhere else could be an indication that there are solutions of \eqref{reduced} which blowup everywhere in $[0,1]$ in both directions of infinity.


\vskip.1in

Set
\begin{equation}
\label{data}
\rho_0(x)=\sin^2\left(2\pi x\right).
\end{equation}
We look for a particular solution of
\begin{equation}
\label{otherode2}
\ddot\mu(t,x)+I(t)\mu(t,x)=\rho_0(x)
\end{equation}
of the form
\begin{equation}
\label{globalform}
\mu(t,x)=\mu_1(t)+\rho_0(x)\mu_2(t),
\end{equation}
with $\mu(0,x)\equiv1$ and $\dot\mu(0,x)=-f_0'(x)$. In \eqref{globalform}, $\mu_1$ and $\mu_2$ satisfy
\begin{equation}
\label{odes}
\ddot\mu_1+I(t)\mu_1=0,\qquad\quad\qquad\quad \ddot\mu_2+I(t)\mu_2=1
\end{equation}
with $\mu_1(0)=1$ and $\mu_2(0)=0$, which are required for $\mu(0,x)\equiv1$ to hold. Now, due to \eqref{meanone},
\begin{equation}
\label{globalmean}
1\equiv\int_0^1{\frac{dx}{\mu_1(t)+\rho_0(x)\mu_2(t)}}.
\end{equation}
Then \eqref{data} yields the relation
\begin{equation}
\label{relation}
\mu_2=\frac{1}{\mu_1}-\mu_1.
\end{equation}
Note that differentiating the above, setting $t=0$ and using $\mu_1(0)=1$, gives $\dot\mu_2(0)=-2\dot\mu_1(0)$. Thus, since
\begin{equation}
\label{icglobal}
-f_0'(x)=\dot\mu(0,x)=\dot\mu_1(0)+\rho_0(x)\dot\mu_2(0),
\end{equation}
if we choose $\dot\mu_1(0)=0$, then $f_0'(x)\equiv0$. So for the time being we simply set
\begin{equation}
\label{N0}
\dot\mu_1(0)=N_0\in\mathbb{R}^+\cup\{0\}.
\end{equation}
Next, using \eqref{relation} to eliminate $I(t)$ in \eqref{odes}, we obtain, after simplification,
\begin{equation}
\label{childress}
(\ln\mu_1)^{\ddot{}}=-\frac{1}{2}\mu_1.
\end{equation}
Then, dividing both sides of \eqref{childress} by $\mu_1$, differentiating in time, and setting
$$N(t)=\frac{\dot\mu_1}{\mu_1}$$
leads to
\begin{equation}
\label{childress2}
2\dot N=N^2-C_0
\end{equation}
for $C_0=1+N_0^2$. Solving \eqref{childress2} yields
\begin{equation}
\label{mu1}
\mu_1(t)=C_0\left[\sqrt{C_0}\cosh\left(\frac{\sqrt{C_0}}{2}t\right)-N_0\sinh\left(\frac{\sqrt{C_0}}{2}t\right)\right]^{-2},
\end{equation}
from which a solution of \eqref{otherode2} can be obtained via \eqref{globalform}, \eqref{relation} and \eqref{mu1}. Note that $f_0'(x)=-N_0\cos(4\pi x)$. Consequently, if we use Dirichlet boundary conditions, or assume $f_0(x)$ to be odd through $x=0$, or simply enforce the mean-zero condition \eqref{meanzero} for periodic initial data, then for the simplest case $N_0=0$, we have that $f_0(x)\equiv0$ and
\begin{equation}
\label{mu1N00}
\gamma_x(t,x)=\left[\text{sech}^2\left(\frac{t}{2}\right)+\frac{1}{2}\left(3+\cosh t\right)\tanh^2\left(\frac{t}{2}\right)\rho_0(x)\right]^{-1}.
\end{equation}
In this case, the global solution corresponding to $f_0(x)\equiv0$ and $\rho_0(x)=\sin^2(2\pi x)$ is obtained from \eqref{jacobian0} and \eqref{rho} as
\begin{equation}
\label{mu1N000}
f_x(t,x)=\cos(4\pi x)\tanh\left(\frac{t}{2}\right),\qquad\qquad
\rho(t,x)=\frac{(1+\cosh t)\rho_0(x)}{2+\left(3+\cosh t\right)\sinh^2\left(\frac{t}{2}\right)\rho_0(x)}.
\end{equation}
More generally, for $N_0>0$, $f_0'(x)=-N_0\cos(4\pi x)$ attains its maximum at $x_1=1/4$ and $x_2=3/4$, with zero initial vorticity at both of these locations. Define $\Lambda\equiv\{0,1/2,1\}$, the zeros of $\rho_0(x)=\sin^2(2\pi x)$. Then as $t\to+\infty$, $\gamma_x(t,x)\to+\infty$ on $\Lambda$ but vanishes everywhere else, $\rho(t,\gamma(t,x))$ vanishes exponentially for all $x\in[0,1]\backslash\Lambda$ and is identically zero on $\Lambda$, whereas, for $x\in[0,1]\backslash\Lambda$ or respectively $x\in\Lambda$, $f_x(t,\gamma(t,x))$ converges to $\sigma(N_0)$ or $-\sigma(N_0)$, where
\begin{equation}
\label{uxconv}
\sigma(N_0)=\frac{1+N_0^2-N_0\sqrt{1+N_0^2}}{N_0-\sqrt{1+N_0^2}}.
\end{equation}
We remark that the behavior described above has been observed in 2d Boussinesq with diffusion (\cite{Li}) and stagnation-point form solutions of the incompressible 2d Euler equations (\cite{Sarria}).


\section{Conclusions}
\label{conclusions}

We presented a local well-posedness result and a regularity criterion for solutions of \eqref{reduced}-\eqref{dbc}, as well as \eqref{reduced} with \eqref{pbc} and mean-zero \eqref{meanzero}. The former can be viewed as an analogue of the well-known regularity criteria for the inviscid 2d Boussinesq equations in terms of the gradient of the velocity field. Using Dirichlet boundary conditions \eqref{dbc}, we also established general criteria for finite-time blowup (from smooth nontrivial initial data) of the time integral of $f_x(t,x)$ at the boundary and, as a consequence, proved one or two-sided blowup in the vorticity \eqref{vorticity}. Assuming $f_0'$ attains its greatest value $M_0>0$ only at the boundary, our blowup criteria makes use of the local profile of $f_0$, as characterized by the non-vanishing of the initial vorticity at the boundary, and a non-negative initial temperature $\rho_0$. Lastly, we constructed an infinite family of solutions to \eqref{reduced} that illustrates how the vanishing of the initial vorticity at, at least, one of the points where $M_0$ is attained (be this point located at the boundary or in the interior), may suppress finite-time blowup. If we restrict the class of initial data to smooth functions satisfying the Dirichlet boundary condition \eqref{dbc}, or periodic boundary condition \eqref{pbc} with mean-zero \eqref{meanzero}, then our results indicate that only \eqref{dbc} may induce finite-time blowup.

\section*{Acknowledgments}

The Authors would like to thank the Referee for helpful suggestions. A. Sarria would like to thank Prof. Stephen C. Preston for discussions. J. Wu was partially supported by NSF grant DMS1209153
and by the AT$\&$T Foundation at Oklahoma State University.

\end{document}